\documentclass[a4paper,11pt,reqno]{amsart}

\usepackage[utf8]{inputenc}

\usepackage{etex}

\usepackage{xcolor}

\definecolor{verydarkblue}{rgb}{0,0,0.5}

\usepackage[
    breaklinks,
    colorlinks,
    citecolor=verydarkblue,
    linkcolor=verydarkblue,
    urlcolor=verydarkblue,
    pagebackref=true,
    hyperindex
]{hyperref}

\backrefenglish

\usepackage{fancyhdr}

\usepackage[
    hscale=0.7,
    vscale=0.75,
    headheight=13pt,
    centering,
]{geometry}

\usepackage{amsmath}
\usepackage{amsthm}
\usepackage{amssymb}
\usepackage{mathtools}
\usepackage{stmaryrd}
\usepackage{mathdots}
\usepackage{framed}
\usepackage[capitalize]{cleveref}
\usepackage{array}
\usepackage[alphabetic]{amsrefs}
\usepackage[all,cmtip]{xy}
\usepackage{enumitem}
\usepackage{calligra,mathrsfs}

\theoremstyle{plain}

\newtheorem{introtheorem}{Theorem}

\crefname{introtheorem}{Theorem}{Theorems}

\crefname{introconjecture}{Conjecture}{Conjectures}

\newtheorem{theorem}{Theorem}[section]

\newtheorem{lemma}[theorem]{Lemma}
\newtheorem{corollary}[theorem]{Corollary}

\newtheorem{conjecture}[theorem]{Conjecture}

\theoremstyle{definition}

\newtheorem{definition}[theorem]{Definition}

\theoremstyle{remark}

\newtheorem{remark}[theorem]{Remark}

\newtheorem{example}[theorem]{Example}

\numberwithin{figure}{section}

\numberwithin{equation}{section}

\IfFileExists{./article-style.tex}{\input{article-style.tex}}{}

\usepackage{marginnote}

\def\Q{{\mathbb Q}}
\def\R{{\mathbb R}}

\def\P{{\mathbb P}}

\def\cA{\mathcal{A}}

\def\cF{\mathcal{F}}

\def\cH{\mathcal{H}}

\def\cL{\mathcal{L}}

\def\cN{\mathcal{N}}

\def\O{\mathcal{O}}

\def\a{\alpha}
\def\b{\beta}

\def\d{\delta}
\def\f{\phi}

\def\p{\pi}

\def\s{\sigma}

\def\.{\cdot}
\def\^{\widehat}
\def\~{\widetilde}
\def\o{\circ}
\def\ov{\overline}

\def\rat{\dashrightarrow}

\def\inj{\hookrightarrow}

\def\({\left(}
\def\){\right)}

\def\*{{}^*}

\renewcommand{\and}{ \ \ \text{ and } \ \ }

\def\sm{\mathrm{sm}}

\DeclareMathOperator{\codim} {codim}

\DeclareMathOperator{\NE} {NE}
\DeclareMathOperator{\CNE} {\ov{\NE}}

\DeclareMathOperator{\Pic} {Pic}

\DeclareMathOperator{\Sing} {Sing}

\DeclareMathOperator{\NS} {NS}

\DeclareMathOperator{\Supp} {Supp}

\DeclareMathOperator{\Hom} {Hom}

\DeclareMathOperator{\Locus} {Locus}

\def\bir{\mathrm{bir}}
\DeclareMathOperator{\RC} {RC}

\DeclareMathOperator{\sheafhom}{\mathscr{H}\text{\kern -3pt {\calligra\large om}}\,}

\def\ratquot{/ \hskip-2.5pt /}

\makeatletter
\@namedef{subjclassname@2020}{%
  \textup{2020} Mathematics Subject Classification}
\makeatother

\begin{document}

\title{Extending rationally connected fibrations from ample subvarieties}

\dedicatory{In memory of Mauro Beltrametti}

\author{Tommaso de Fernex}

\address{Department of Mathematics, University of Utah, Salt Lake City, UT 48112, USA}

\email{defernex@math.utah.edu}

\author{Chung Ching Lau}

\email{malccad@gmail.com}


\thanks{%
The research of the first author was partially supported by NSF Grant DMS-1700769
and by NSF Grant DMS-1440140 while in residence at  
MSRI in Berkeley during the Spring 2019 semester.
The research of the second author was partially supported by a Croucher Foundation Fellowship. 
}

\begin{abstract}
Using deformation theory of rational curves, we prove a conjecture
of Sommese on the extendability of morphisms from ample subvarieties
when the morphism is a smooth (or mildly singular) fibration with rationally connected fibers.
We apply this result in the context of Fano fibrations and
prove a classification theorem for projective bundle and quadric fibration
structures on ample subvarieties.
\end{abstract}

\maketitle

\section{Introduction}

This paper is motivated by the following conjecture. 

\begin{conjecture}[Sommese \cite{Som76}]
\label{conj-Som-intro}
Let $X$ be a smooth complex variety and $Y \subset X$ a smooth subvariety of codimension $r$ defined 
by a regular section of an ample vector bundle on $X$. 
Then any morphism $\p \colon Y \to Z$ with $\dim Y - \dim Z > r$
extends to a morphism $\~\p \colon X \to Z$. 
\end{conjecture}

The main purpose of \cite{Som76} was to analyze restrictions for 
a projective manifold $Y$ to be an ample divisor in a projective manifold $X$, 
a setting that generalizes more classical studies on hyperplane sections. 
It is in this context that Sommese proved, among other things, 
that the $r=1$ case of \cref{conj-Som-intro} holds, namely, that
if $Y$ is a smooth ample divisor in a smooth complex projective variety $X$,
then any morphism $\p \colon Y \to Z$ with $\dim Y - \dim Z > 1$
extends to a morphism $\~\p \colon X \to Z$.
Letting $Y$ be defined by a regular section 
of an ample vector bundle on $X$ is a natural way of extending the setting to higher codimensions.
This setting is briefly discussed in the second appendix of \cite{Som76}, 
at the end of which that \cref{conj-Som-intro} was stated.

When $Y$ is an ample divisor of $X$, 
the conjecture being already known in this case, people have further investigated 
the boundary case where $\p$ has relative dimension one (the first case beyond the 
bound imposed in the statement), see
\cite{BS95,BI09,Lit17,Liu19} and the references therein. 
By contrast, very little is known about \cref{conj-Som-intro} when $r > 1$.

The setting considered in the conjecture, where $Y$ is a smooth variety defined by a regular section 
of an ample vector bundle on a smooth complex variety $X$, was later revived 
in a series of papers, starting with \cite{LM95}, whose aim was 
to understand how much constraint
the geometry of $Y$ imposes on the geometry of $X$ when $Y$ is assumed to be special from the point of 
view of adjunction theory. 
Some of the results obtained in that period 
deal with situations where $Y$ is equipped with a morphism $\p \colon Y \to Z$, 
most of the times a Mori contraction of some type, see, e.g., \cite{LM96,AO99,dFL99,dF00,LM01,ANO06,Occ06,BdFL08}. 
In some cases such morphism is shown to extend to $X$.
Perhaps the strongest evidence toward Sommese's conjecture 
coming out of these works can be found in \cite{BdFL08}, whose main result can be viewed as a
`rational' solution of the conjecture in the context of rationally connected fibrations
which is there applied to verify the conjecture
for all projective bundles and quadric fibrations of relative Picard number one.

In this paper, we address \cref{conj-Som-intro} when $\p$ is 
a morphism with rationally connected fibers.
In a separate paper, \cite{dFL}, we use different techniques to verify the conjecture 
in some other cases where $X$, $Y$, or $\p$ are built up from toric and abelian varieties.

We consider Sommese's conjecture in the more general context of 
ample subvarieties, a notion due to Ottem \cite{Ott12} that
was not available at the time of the writing of \cite{Som76}
but nonetheless fits very naturally (see \cref{def:ample}
and \cref{conj:Som} below). This is the same setting considered in \cite{dFL}. 
We should stress that the condition of being an ample subvariety is far less restrictive
than being defined by a regular section of an ample vector bundle.

A special case of our main theorem gives the following extension property. 

\begin{introtheorem}[cf.\ \cref{th:Mori-contr}]
\label{th-intro}
Let $X$ be a smooth complex projective variety
and $Y \subset X$ a smooth ample subvariety of codimension $r$.
Let $\p \colon Y \to Z$ be a smooth morphism with rationally connected fibers
such that $\dim Y - \dim Z > r$. Then $\p$
extends uniquely to a morphism $\~\p \colon X \to Z$.
\end{introtheorem}

The assumption that $\p$ is smooth can be relaxed, and we refer to 
\cref{th:Mori-contr} for a stronger result.  
The proof builds on the main result of \cite{BdFL08} which 
provides us with a rational map $\~\p \colon X \rat Z$ extending $\p$.
The main contribution of the paper is to prove that
this rational map is a well-defined morphism. 
A key ingredient in the proof is a version of Grothendieck--Lefschetz theorem 
for ample subvarieties from \cite{dFL}, which allows us
to go beyond the case where $Y$ is assumed to be defined by a regular section of an ample vector bundle. 
\cref{th-intro} applies, for instance, to smooth Mori contractions, and just like
before the smoothness assumption can be relaxed.

As an application of \cref{th-intro}, 
we verify the conjecture for all fibrations in Fano complete intersections
of index larger than the codimension of complete intersection (see \cref{def:Fano-ci-fibr,th:P-fibr-Q-fibr}). 
In particular, this proves the conjecture (in the more general form stated in \cref{conj:Som})
when $\p$ is a projective bundle or a quadric fibration,
two cases that were previously investigated under more restrictive conditions.

Under an additional condition on the Picard groups
(a condition that is unnecessary if one assumes that $Y$ is defined by a regular section of an ample vector bundle), 
we obtain the following structure theorem.

\begin{introtheorem}[cf.\ \cref{cor:scroll-quadric}]
\label{cor:scroll-quadric-intro}
Let $X$ be a smooth complex projective variety
and $Y \subset X$ a smooth ample subvariety of codimension $r$.  
Assume that the restriction map $\Pic(X) \to \Pic(Y)$ is surjective, 
and let $\p \colon Y \to Z$ be a projective bundle or
a quadric fibration with integral fibers,
such that $\dim Y - \dim Z > r$. 
Then $\p$ extends uniquely to a projective bundle or a
quadric fibration with integral fibers $\p \colon X \to Z$, and the 
fibers of $\~\p$ embed linearly in the fibers of $\~\p$.
\end{introtheorem}

This theorem improves upon several earlier results.
When $r=1$, the case where $\p$ is a projective bundle follows
from \cite[Proposition~III]{Som76} and \cite[Theorem~5.5]{BI09}.
When $Z$ is a curve and $\p$ has relative Picard number 1, 
\cref{cor:scroll-quadric-intro} follows from case~(a) of \cite[Theorem~5.8]{BdFL08}.
As we may well assume that $\dim Y \ge 3$, the statement being trivial otherwise, 
the Lefschetz--Sommese theorem 
shows that the hypothesis that $\Pic(X) \to \Pic(Y)$ is surjective is automatic
if $Y$ is assumed to be defined by a regular section of an ample vector bundle on $X$.
In this more restrictive setting, special cases of \cref{cor:scroll-quadric-intro} were first obtained:
in \cite{LM96} when $\p$ is a projective bundle over a curve of positive genus,
and when $Z$ is a curve and there exists {\it a priori} a polarization of $X$
inducing a relatively linear polarization on $\p$;
in \cite[Theorems~4.1 and~5.1]{AO99} 
with no restrictions on $Z$ but still assuming the existence of such a polarization;
and in case~(b) of \cite[Theorem~5.8]{BdFL08} 
for all projective bundles and quadric fibrations of relative Picard number one.

The proof of \cref{th-intro} uses deformation theory of rational curves and relies, 
in particular, on properties of the scheme $\Hom(\P^1,X)$ 
parameterizing maps $\P^1 \to X$. 
It would be interesting to see whether a more delicate 
analysis of deformation theory of 1-cycles, using
Chow varieties in place of the Hom scheme, might lead to
a proof of the conjecture for all extremal Mori contractions and, more generally, 
for all Mori contractions of pure fiber-type
and all rationally connected fibrations not contracting divisors
(see \cref{def:contr-no-div,def:pure-fiber-type}).

\subsection*{Acknowledgements}

The first author is indepted to Mauro Beltrametti (to whom the paper is dedicated)
and Antonio Lanteri for many fruitful conversations they had over several years
on the topic of the paper.

\section{Ample subvarieties and Sommese's extendability conjecture}

We recall the definitions of $q$-ampleness and ample subscheme from \cite{Tot13,Ott12}. 

\begin{definition}
\label{def:ample}
Given a nonnegative integer $q$, a line bundle $\cL$ on a complex projective variety $X$ 
over a field is said to be \emph{$q$-ample} 
if every coherent sheaf $\cF$ on $X$ we have 
$H^i(X,\cF \otimes \cL^{\otimes m}) = 0$ for all $i > q$ and all $m$ sufficiently large
depending on $\cF$.
The same terminology is used for a Cartier divisor $D$ if the condition is satisfied by $\O_X(D)$. 
A closed subscheme $Y$ of codimension $r > 0$ of a complex projective variety $X$ is said to be 
\emph{ample} if the exceptional
divisor $E$ of the blow-up of $X$ along $Y$ is $(r-1)$-ample.
\end{definition}

Examples of ample subschemes are given by schemes defined (scheme theoretically) by regular sections
of ample vector bundles on complex projective varieties \cite[Proposition~4.5]{Ott12}.
The notion of ample subscheme can be thought of as a generalization of this.
Other examples are given by smooth curves with ample normal bundle in projective homogeneous
varieties \cite[Proposition~8.1]{Ott12}, and smooth subvarieties of projective spaces
whose embeddings satisfy the Lefschetz hyperplane with rational coefficients \cite[Theorem~7.1]{Ott12}.
Notice that many of these examples (for instance, all examples where 
the Lefschetz hyperplane does not hold with integral coefficients) 
cannot be realized as zero sets of regular sections of ample vector bundles.

We will use the following extensions of the
Lefschetz hyperplane theorem and the Grothendieck--Lefschetz theorem to ample subvarieties.

\begin{theorem}[\protect{\cite[Corollary~5.3]{Ott12}}]
\label{th:Lefschetz}
Let $X$ be a smooth complex projective variety, 
and let $Y \subset X$ be an ample l.c.i.\ subscheme. Then the restriction map
$H^i(X,\Q)\to H^i(Y,\Q)$ is an isomorphism for $i < \dim Y$ and is injective for $i = \dim Y$.
\end{theorem}

\begin{theorem}[\protect{\cite[Theorem~A]{dFL}}]
\label{th:Pic}
Let $X$ be a smooth complex projective variety
and $Y \subset X$ be a smooth ample subvariety.
Then the restriction map $\Pic(X) \to \Pic(Y)$ is injective if $\dim Y \ge 2$
and has finite cokernel if $\dim Y \ge 3$.
\end{theorem}

We now come to \cref{conj-Som-intro}. 
As we discussed, subschemes defined by regular sections of ample vector bundles are 
ample in the ambient variety, and we consider the following reformulation of 
Sommese's conjecture in the context of ample subvarieties.

\begin{conjecture}
\label{conj:Som}
Let $X$ be a smooth complex projective variety 
and $Y \subset X$ a smooth ample subvariety of codimension $r$.
Then any morphism $\p \colon Y \to Z$ with $\dim Y - \dim Z > r$
extends uniquely to a morphism $\~\p \colon X \to Z$. 
\end{conjecture}

It is easy to see that the condition that $\dim Y - \dim Z > r$ is sharp. 
When $r=1$, this is discussed in \cite[Section~3]{BI09}, and 
the construction given there can be extended to include the following 
example in arbitrary codimension $r$. 

\begin{example}
\label{eg:sharp}
Let $r,s$ be two positive integers. 
Denoting by $u_0,\dots,u_r \in H^0(\P^r,\O_{\P^r}(1))$ a set of generators, 
consider the exact sequence
\[
0 \to \O_{\P^r}^{\oplus r} \xrightarrow{\a} 
\O_{\P^r}(r+1)^{\oplus r+1} 
\xrightarrow{\b} \O_{\P^r}(2r+1) 
\to 0
\]
where $\b$ is given on global sections by 
\[
\b \colon (s_0,\dots,s_r) \mapsto \sum_{i=0}^r s_i u_i^r,
\]
and $\a$ is given on global sections by
\[
\a \colon (t_1,\dots,t_r) \mapsto \Big( - \sum_{i=1}^r t_iu_i^{r+1},t_1u_0^ru_1,\dots, t_ru_0^ru_r \Big).
\]
Adding a new summand $\O_{\P^r}(2r+1)^{\oplus s}$ to the middle and right terms, 
with the identity map in between, we obtain the exact sequence
\[
0 \to \O_{\P^r}^{\oplus r} \to
\O_{\P^r}(r+1)^{\oplus r+1} \oplus \O_{\P^r}(2r+1)^{\oplus s}
\to \O_{\P^r}(2r+1)^{\oplus s+1} 
\to 0.
\]
Let $X = \P(\O_{\P^r}(r+1)^{\oplus r+1} \oplus \O_{\P^r}(2r+1)^{\oplus s})$
and $Y = \P(\O_{\P^r}(2r+1)^{\oplus s+1})$. 
We have a fiberwise embedding $Y \subset X$ of scrolls over $\P^r$. 
By construction, $Y$ is defined, scheme theoretically, by a regular section of $\O_X(1)^{\oplus r}$, 
where $\O_X(1)$ is the tautological line bundle. Note that this is an ample vector bundle; 
in particular, $Y$ is an ample subvariety of $X$. 
Now, we have $Y \cong \P^r \times \P^s$, and the second projection 
$Y \to \P^s$ does not extend to $X$. Note that this projection 
has relative dimension $r$.
\end{example}

\section{Extending rationally connected fibrations}
\label{s:Mori-contr}

We start by recalling some terminology from \cite{Kol96,BdFL08}
to which we refer for further details and basic properties.
Let $X$ be a smooth complex variety, and
denote by $\Hom_\bir(\P^1,X)$ the scheme parameterizing 
morphisms from $\P^1$ to $X$ that are birational to their images. 

\begin{definition}
An element $[f] \in \Hom_\bir(\P^1,X)$ is said to be a \emph{free rational curve}
(resp., a \emph{very free rational curve}) if $f^*T_X$ is nef (resp., ample). 
A \emph{family of rational curves} on $X$ is 
an arbitrary union of irreducible components of $\Hom_\bir(\P^1,X)$. 
\end{definition}

Let $V$ be a famility of rational curves on $X$. 
If $0 \in \P^1$ is a fixed point and $Z \subset X$ is a closed subscheme, then 
$V(\{0\} \to Z)$ denotes the closed subscheme of $V$ defined by
the condition that, for $[f] \in V$, we have $[f] \in V(\{0\} \to Z)$
if and only if $f(0) \in Z$.
The image of the evaluation map $\P^1 \times V \to X$ is denoted by 
$\Locus(V)$, and $\Locus(V;\{0\} \to Z)$ is defined similarly. 

Assume now that $X$ is projective. 

\begin{definition}
A family of rational curves $V \subset \Hom_{\bir}(\P^1,X)$ is a \emph{covering family} 
if $\Locus(V_i)$ is dense in $X$ for every irreducible component $V_i$ of $V$.
\end{definition}

Associated to every covering family $V$, 
there is a model $X\ratquot_V$ (only defined up to birational equivalence)
and a dominant rational map $\f\colon X \rat X\ratquot_V$
such that $\f$ restricts to a proper morphism over a nonempty open set of $X\ratquot_V$ and
a very general fiber is an equivalence class of an equivalence relation defined by $V$
(see \cite[Section~IV.4]{Kol96} for the precise definition of the equivalence
relation and the construction of $\f$).

\begin{definition}
The map $\f\colon X \rat X\ratquot_V$ is called the \emph{$\RC_V$-fibration} of $X$, 
and $X \ratquot_V$ the \emph{$\RC_V$-quotient}.
The variety $X$ is said to be \emph{$\RC_V$-connected} if $X\ratquot_V$ is a point. 
\end{definition}

Given a closed embedding $\iota \colon Y \inj X$ of a smooth subvariety, there is a
natural map $\iota_* \colon \Hom_\bir(\P^1,Y) \to \Hom_\bir(\P^1,X)$ given by composition.
For a set $S \subset \Hom_\bir(\P^1,X)$ we denote by $\iota^{-1}_*(S) \subset \Hom_\bir(\P^1,Y)$
its inverse image via $\iota_*$, and for a set $T \subset \Hom_\bir(\P^1,Y)$ we denote by 
$\iota_*(T) \subset \Hom_\bir(\P^1,X)$ its image via $\iota_*$. 

\begin{definition}
Given a family of rational curves $V$ on $X$, 
the \emph{restriction} $\rangle\iota^{-1}_*(V)\langle$ of $V$ to $Y$
is defined to be the largest family of rational curves on $Y$ 
that is contained in $\iota^{-1}_*(V)$; equivalently, $\rangle\iota^{-1}_*(V)\langle$ is the union
of all irreducible components of $\iota^{-1}_*(V)$ that are also irreducible 
components of $\Hom_\bir(\P^1,Y)$.
Similarly, given a family of rational curves $W$ on $Y$, 
its \emph{extension} $\langle\iota_*(W)\rangle$ to $X$
is defined to be the union
of all irreducible components of $\Hom_\bir(\P^1,X)$ that contains at least one irreducible 
component of $\iota_*(W)$.
\end{definition}

The main result of \cite{BdFL08}, which is recalled below, can be seen as providing 
a `rational' solution to \cref{conj:Som} in the context of rationally connected fibrations. 
A related result which applies to maximal rationally connected fibrations
over bases of positive geometric genus was also obtained in \cite{Occ06}.

\begin{theorem}[\protect{\cite[Theorem~3.6]{BdFL08}}]
\label{th:BdFL}
Let $X$ be a smooth complex projective variety
and $Y \subset X$ a smooth ample subvariety of codimension $r$. 
Denote by $\iota \colon Y \inj X$ the inclusion map.
Let $V \subset \Hom_\bir(\P^1,X)$ be a family of rational curves, 
and assume that the restriction to $Y$ of every irreducible component 
of $V$ is a covering family of rational curves on $Y$.
Let $V_Y := \rangle\iota^{-1}_*(V)\langle$ be the restriction of $V$ to $Y$, 
and let $\a \colon X \rat X\ratquot_V$ and $\b \colon Y \rat Y\ratquot_{V_Y}$ denote the respective rationally connected fibrations. 
Assume that $\dim Y - \dim Y\ratquot_{V_Y} > r$. 
Then there is a commutative diagram
\[
\xymatrix{
Y \ar@{^(->}[r]^\iota \ar@{-->}[d]_\b & X \ar@{-->}[d]^{\a} \\
Y\ratquot_{V_Y} \ar@{-->}[r]^\d & X\ratquot_V
}
\]
where $\d$ is a birational map. 
\end{theorem}

\begin{remark}
The statement of \cref{th:BdFL} is actually a slight variation of \cite[Theorem~3.6]{BdFL08}.
The original statement in \cite{BdFL08} imposes a weaker condition on $Y$, 
only requiring that the normal bundle $\cN_{Y/X}$ is ample and the induced map
on N\'eron--Severi spaces $N^1(X) \to N^1(Y)$ is surjective
(here, $N^1(X) = \NS(X)_\R$); however, 
the conclusion is also weaker, namely, that the map $\d$ is dominant and generically finite. 
By assuming that $Y$ is an ample subvariety (which implies that $\cN_{Y/X}$ is ample
and $N^1(X) \to N^1(Y)$ is surjective), we can conclude that $\d$ must in fact be birational. 
To see this, let $s \in X\ratquot_V$ be a general point, and let $X_s$ and $Y_s$ be the 
fibers over $s$.
Here we assume that the fiber $\d^{-1}(s)$ is a finite
set of cardinality equal to the degree of $\d$. 
Note that $X_s$ is smooth and connected. 
By \cite[Proposition~4.8]{Lau16}, $Y_s$ is a positive dimensional ample subvariety of $X_s$, 
and therefore it is connected since, by \cref{th:Lefschetz}, 
the map $H^0(X_s,\Q) \to H^0(Y_s,\Q)$ is an isomorphism. 
This implies that $\d$ is birational. 
\end{remark}

As an application of the above \lcnamecref{th:BdFL}, 
\cref{conj:Som} was verified in \cite{BdFL08} when $\p\colon Y \to Z$ is a projective bundle
or a quadric fibration with integral fibers and relative Picard number 1, 
assuming that either $Y$ is defined by a regular section of an ample vector bundle on $X$
(as in the original conjecture of Sommese), or that $Z$ is a curve. 
We refer to the introduction of \cite{BdFL08} 
for quick overviews of other related results in the literature.

Under some conditions on the fibers which we discuss next,
we apply \cref{th:BdFL} to prove \cref{conj:Som} for 
Mori contractions and, more generally, fibrations with rationally connected fibers.

\begin{definition}
\label{def:RC-fibers}
A surjective morphism of varieties $\f \colon X \to Y$ is said to have
\emph{rationally connected fibers} if a general fiber of $\f$ is rationally connected
(or, equivalently, if every fiber is rationally chain connected).
\end{definition}

\begin{definition}
\label{def:contr-no-div}
A surjective morphism of varieties $\f \colon X \to Y$ is said 
\emph{not to contract divisors} if there are no prime divisors
$D$ in $X$ such that $f(D)$ has codimension $\ge 2$ in $Y$.
\end{definition}

\begin{definition}
\label{def:pure-fiber-type}
We say that a Mori contraction $\f \colon X \to Y$ 
is \emph{of fiber-type} if all fibers are positive dimensional, 
and that it is of \emph{pure fiber-type} if every extremal ray of the face
of the Mori cone $\CNE(X) \subset N_1(X)$ 
contracted by $\f$ defines an extremal Mori contraction of fiber-type. 
\end{definition}

For example, any extremal Mori contraction of fiber-type is of pure fiber-type, and
a conic bundle over a curve admitting reducible fibers is a contraction of fiber-type
but not of pure fiber-type.

\begin{theorem}
\label{th:Mori-contr}
Let $X$ be a smooth complex projective variety and
$Y \subset X$ a smooth ample subvariety of codimension $r$.
Let $\p \colon Y \to Z$ be a surjective morphism with $\dim Y - \dim Z > r$, 
and assume that either
\begin{enumerate}
\item
\label{item:pi-pure-fiber-type}
$\p$ is a Mori contraction of pure fiber-type, or
\item
\label{item:pi-contr-no-div}
$\p$ does not contract divisors. 
\end{enumerate}
Assume furthermore that there exists an open set $Z^* \subset Z$ with complement of codimension $\ge 2$
such that for every $z \in Z^*$
the fiber $Y_z$ is irreducible and contains in its smooth locus a very free rational curve.
Then $\p$ extends uniquely to a morphism $\~\p \colon X \to Z$.
\end{theorem}

\begin{proof}
Since the statement is trivial if $Z$ is a point, we can assume that $\dim Z \ge 1$, 
and hence $\dim Y \ge 3$. 
By \cref{th:Pic}, the inclusion $\iota \colon Y \inj X$
induces an isomorphism $\iota^* \colon N^1(X) \to N^1(Y)$ and, by duality, 
an isomorphism $\iota_* \colon N_1(Y) \to N_1(X)$. 

Let $F$ be a fiber of $\p$ of dimension $\dim F = \dim Y - \dim Z$
such that the smooth locus $F_\sm$ of $F$ contains a very free rational curve
$h \colon \P^1 \to F_\sm$. Note that any fiber over $Z^*$ will satisfy this condition.
Let $U \subset Y$ be an open set containing the image of $h$ and such that
$F \cap U \subset F_\sm$. 
We see by the splitting of the exact sequence
\[
0 \to h^*T_{F \cap U} \to h^*T_U|_{F \cap U} \to h^*N_{F \cap U/U} \to 0
\] 
that $h$ defines, by composition with
the inclusion of $F$ in $Y$, a free rational curve $f \colon \P^1 \to Y$.
If $W$ is the irreducible component of $\Hom_\bir(\P^1,Y)$ containing $[f]$, 
then $W$ is a covering family of rational curves on $Y$. 
Note that $\R_{\ge 0}[W] \subset N_1(Y)$ is contained in the extremal face
of $\CNE(Y)$ contracted by $\p$, and this means that the latter, viewed 
as a rational map, factors through the $\RC_W$-fibration $Y \rat Y\ratquot_{W}$. 
As $F$ is $\RC_W$-connected, we conclude that these two maps have 
the same very general fibers and
hence $\p$ agrees, as rational maps, with the $\RC_W$-fibration. 

Let $V := \langle i_*(W) \rangle \subseteq \Hom_{\bir}(\P^1,X)$
be the extension of $W$ to $X$, and consider
the restriction $V_Y := \;\rangle i_*^{-1}(V) \langle\; \subseteq \Hom_{\bir}(\P^1,Y)$
of $V$ to $Y$. By \cite[Proposition~3.11]{BdFL08}, the
$\RC_{V_Y}$-fibration $Y \rat Y\ratquot_{V_Y}$
agrees with the $\RC_W$-fibration of $Y$ and hence with
the contraction $\p$. 
Note also that 
\[
\R_{\ge 0}[V_Y] = \R_{\ge 0}[V] = \R_{\ge 0}[W]
\]
via the identification $\iota_* \colon N_1(Y) \cong N_1(X)$.

Let $\~\p \colon X \rat X\ratquot_V$ be the $\RC_V$-fibration. 
The models $X\ratquot_V$ and $Y\ratquot_W$ are 
defined up to birational equivalence, but $Z$, which is a model for $Y\ratquot_W$, 
is uniquely determined, up to isomorphism, by the contraction $\p$. 
By \cref{th:BdFL}, $X\ratquot_V$ is birational to $Z$, thus we have a commutative diagram
\[
\xymatrix{
Y \ar@{^(->}[r]^\iota \ar[d]_\p & X \ar@{-->}[dl]^{\~\p} \\
Z & 
}.
\]

Fix an embedding $Z \subset \P^m$. 
Let $\cA = \O_{\P^m}(1)|_Z$, and let $\cL$ be a line bundle
whose global sections define the rational map $X \rat \P^m$. 
If $p \colon X' \to X$ is a proper birational morphism
such that $q := \~\p \o p \colon X' \to Z$ is a morphism, then
we have $\cL \cong \O_X(p_*q^*A)$ for any $A \in |\cA|$.
Our goal is to show that $\cL|_Y \cong \p^*\cA$.

By construction, $\~\p$ is defined by a linear subsystem $|\Lambda|$ of $|\cL|$, where 
$\Lambda \subset H^0(X,\cL)$ is a subspace. 
Let $B \subset X$ denote the base scheme of $|\Lambda|$. 
Note that the support of $B$ is the indeterminacy locus of $\~\p$. 
To prove the theorem, we need to show that $B = \emptyset$. 
This will show that $\cL|_Y \cong \p^*\cA$, hence that $\f$ is a morphism giving the desired extension of $\p$. 

Suppose by contradiction that $B \ne \emptyset$. 
Then
\[
\dim B \ge \dim X - \dim Z - 1.
\]
This is proved in \cite{Ste68}. Alternatively, 
one can see this directly by taking a general linear projection $\P^m \rat \P^{\dim Z}$.
Since the induced map $Z \to \P^{\dim Z}$ is a morphism, it follows
that the indeterminacy locus of $\~\p$ is the same as the one of its composition 
with the projection to $\P^{\dim Z}$, and hence $B$ is cut out, 
set theoretically, by $\dim Z + 1$ divisors. This implies the lower-bound on $\dim B$ stated above. 

Since $Y$ is ample in $X$ and $\dim Y + \dim B \ge \dim X$, it follows that 
\[
B \cap Y \ne \emptyset.
\]
Let $\Lambda_Y \subset H^0(Y,\cL|_Y)$ be
the image of $\Lambda$ under restriction map $H^0(X,\cL) \to H^0(Y,\cL|_Y)$.
The commutativity of the above diagram 
implies that $B$ cuts, scheme theoretically, a nonempty effective Cartier divisor $E$ on $Y$
such that
\[
|\Lambda_Y| = |\p^*\cA| + E.
\]
Note that, in particular, $\cL|_Y \cong \p^*\cA \otimes \O_Y(E)$.

We claim that $\p(\Supp(E))$ has codimension one in $Z$. 
This is clear if $\p$ satisfies the condition given in \eqref{item:pi-contr-no-div}
in the statement of the \lcnamecref{th:Mori-contr}.
Suppose then that $\p$ satisfies \eqref{item:pi-pure-fiber-type}. 
In this case, every irreducible curve $C$ in $Y$ that is contracted
by $\p$ is numerically equivalent to a multiple of a curve $C'$ that is
contained in a general fiber of $\p$. Since $\~\p$ restricts to a proper
morphism $X^\o \to Z^\o$ for some nonempty open subset $Z^\o \subset Z$,
it follows that $\cL|_Y\.C = 0$. As clearly $\p^*\cA\.C$, 
this gives $E_Y\.C = 0$, hence it follows by the cone theorem that $\O_Y(E)$ is the pull-back
of a line bundle on $Z$. This means that $E$ is the pull-back of a Cartier divisor on $Z$,
and therefore $\p(\Supp(E))$ has codimension one in $Z$, as claimed.

Recall the assumption stated in the \lcnamecref{th:Mori-contr} 
on the fibers of $\p$ over the open set $Z^* \subset Z$. 
Since the complement of $Z^*$ has codimension $\ge 2$, 
it follows by the above claim that 
\[
\p(\Supp(E)) \cap Z^* \ne \emptyset.
\]
Since the fibers of $\p$ over $Z^*$ are irreducible, 
$E$ must contain in its support a fiber $F := Y_z$ over a point $z \in Z^*$,
and such fiber contains a very free rational curve $h \colon \P^1 \to F_\sm$ within its smooth locus.
By composing with the inclusion of $F$ in $Y$, 
this yields a free rational curve $f \colon \P^1 \to Y$ supported in $F_\sm$. 
We may assume without loss of generality that this fiber $F$ is the same as the fiber picked at the
beginning of the proof, and that the maps $h$ and $f$ are also the same. 
Let $g := \iota \o f \colon \P^1 \to X$. Note that $[g] \in V$. 

Pick an irreducible component $V'$ of $V$ that contains $[g]$. 
Note that $[g] \in V'(\{0\} \to Y)$ and, in fact, $[g] \in V'(\{0\} \to E)$. 
The argument of \cite[Lemma~3.4]{BdFL08} shows that $V'(\{0\} \to Y)$ is smooth at $[g]$ and that the evaluation map
\[
\P^1\times V'(\{0\}\to Y) \to X
\] 
has full rank, equal to $\dim X$, at $(q,[g])$ where $q$ is any point in $\P^1\setminus\{0\}$. 
Its restriction to $V'(\{0\}\to E) \subset V'(\{0\}\to Y)$, namely, the evaluation map
\[
\P^1\times V'(\{0\}\to E) \to X,
\] 
has rank $\ge \dim X - 1$ at $(q,[g])$. 
This follows from the fact that the subscheme $V'(\{0\}\to E) \subset V'(\{0\}\to Y)$ is cut out by one equation,
locally at $[g]$. Indeed, we have the following fiber diagram
\[
\xymatrix{
V'(\{0\}\to E) \ar@{}[rd]|-(.43)\square \ar[d]\ar@{^(->}[r] & V'(\{0\}\to Y) \ar[d] \\
E \ar@{^(->}[r] & Y
}
\]
where the vertical arrows send any element $[g']$ to $g'(0)$, and $E$ is locally
cut out by one equation in a neighborhood of $g(0)$.
Therefore $\Locus(V'(\{0\}\to E))$ has dimension at least $\dim X - 1$. 

To conclude, it suffices to show that under our assumption that $B \ne \emptyset$, we have  
\[
\Locus(V';\{0\}\to E) \subset B.
\]
This will contradict the fact that the indeterminacy locus of a rational map 
on a normal variety must have codimension $\ge 2$, thus finishing the proof.

The above inclusion follows from the following observation.
Let $C$ be an irreducible curve on $X$ with numerical class in 
$\R_{\ge 0}[V]$. Recall that this cone is the image of $\R_{\ge 0}[W]$ under 
the isomorphism $\iota_* \colon N_1(Y) \cong N_1(X)$. 
Using again that $\~\p$ restricts to a proper
morphism $X^\o \to Z^\o$ and $W$ is a covering family, 
we see that $\cL\.C = 0$. 
This implies that for any such curve $C$ we have that either $C \cap B = \emptyset$
or $C \subset B$. 
Now, since every curve parameterized by an element of $V'(\{0\}\to E)$
meets $E$ and hence $B$, it follows that $\Locus(V';\{0\}\to E)$ must be fully contained in $B$.
\end{proof}

\section{Extending Fano fibrations}
\label{s:Fano-fibr}

\cref{th:Mori-contr} can be used to settle \cref{conj:Som} for fibrations
in Fano complete intersections of index larger than the codimension of complete
intersection.

Recall that a morphism of varieties $\p \colon Y \to Z$ is a \emph{projective bundle} 
(or \emph{$\P^n$-bundle}, if $n$ is the relative dimension)
if it is locally of the form $U \times \P^n \to U$, with $U \subset Z$ open, and
the transition functions are linear. 
If $Z$ is smooth, then every projective bundle over $Z$ is isomorphic to the projectivization of a 
locally free sheaf on $Z$ \cite[Exercise~II.7.10]{Har77} and therefore admits a polarization $\cH$
inducing a linear polarization on the fibers. 

\begin{definition}
\label{def:Fano-ci-fibr}
We say that a flat morphism $\p \colon Y \to Z$ 
of relative dimension $n \ge 1$ is a \emph{fibration in Fano complete intersections} 
if there exists a $\P^{n+c}$-bundle $\p' \colon Y' \to Z$
and a fiberwise embedding $Y \inj Y'$ over $Z$
such that the general fiber of $\p$ is a Fano variety and every fiber of $\p$ is
embedded as a (possibly singular) nondegenerate complete intersection of codimension $c$ in the corresponding
fiber of $\p'$. The number $c$ is called the \emph{codimension} of $\p$, and the
\emph{index} of $\p$ is the Fano index of a general fiber. 
\end{definition}

If $n$ is the relative dimension of $\p$ and $d_1,\dots,d_c$ are the degrees of the equations
cutting the fibers of $\p$ in the fibers of $\p'$, then the index is 
given by $n+c + 1 - \sum d_i$, 
with the only exception when $\p$ is a conic bundle, which has index 2 and not 1. 

Special cases of fibrations in Fano complete intersections include
projective bundles, which correspond to the case $c=0$, 
and \emph{quadric fibrations}, which correspond to the case $c=1$ and $d_1 = 2$. 
The following result implies, in particular, 
that \cref{conj:Som} holds for all projective bundles and quadric fibrations. 

\begin{corollary}
\label{th:P-fibr-Q-fibr}
Let $X$ be a smooth complex projective variety 
and $Y \subset X$ a smooth ample subvariety of codimension $r$.
Assume that $\p \colon Y \to Z$ is a fibration in Fano complete intersections
of codimension $c \ge 0$ and index $> c$, with $\dim Y - \dim Z > r$.
Then $\p$ extends uniquely to a morphism $\~\p \colon X \to Z$.
\end{corollary}

\begin{proof}
The statement is trivial if $Z$ is a point, so we can assume that $\dim Z \ge 1$.
Hence $\p$ has relative dimension $n \ge 2$. 
Note that $\p$ satisfies the condition 
of \cref{th:Mori-contr} given in \eqref{item:pi-contr-no-div}.
Then the \lcnamecref{th:P-fibr-Q-fibr} follows from \lcnamecref{th:Mori-contr}
once we verify the condition on the fibers of $\p$ on a suitable open
set $Z^* \subset Z$ stated in the \lcnamecref{th:Mori-contr}.

We fix a fiberwise embedding of $Y$ into a $\P^{n+c}$- bundle $\p' \colon Y' \to Z$
as in the definition. 

By taking $c$ general hyperplane sections, one sees 
that every Fano complete intersection $V \subset \P^{n+c}$ of dimension $n \ge 2$, codimension $c \ge 0$,
and index $> c$ contains a very free rational curve
in its smooth locus, provided the singular locus of $V$ has dimension $< c$. 
So, all we need to check is that, away from a set of codimension $\ge 2$
in the base $Z$, the fibers of $\p$ have singular locus of dimension $< c$. 

This can be checked by restricting $\p$ over a general
complete intersection curve $B \subset Z$. Set $W := \p^{-1}(B)$ and $W' := (\p')^{-1}(B)$, 
and let $\p|_W \colon W \to B$ and $\p'|_{W'} \colon W' \to B$
be the restrictions of $\p$ and $\p'$. By Bertini, we can assume that $W$, $W'$ and $B$ are all smooth.
As the fibers of $\p|_W$ have dimension $\ge 2$, a 
local computation of the equations of $W$ in $W'$ then shows that the presence
of fibers of $\p|_W$ with singular locus of dimension $\ge c$ would confute the smoothness of $W$.

To see this, assume by contradiction that $\p|_W$ has a fiber $F$ with singular locus 
of dimension $\ge c$. Let $t$ be a local parameter on $B$ centered at the base $p$ of the fiber, 
and let $(x_0:\dots:x_{n+c})$ be homogeneous coordinates of $\P^{n+c}$, where $n$ is the relative dimension 
of $\p$. We can assume that $W$ is defined in a local trivialization $U \times \P^{n+c}$ of $W'$ 
by the equations $f_i + tg_i = 0$, for $1 \le i \le c$, where
$f_i \in k[x_0,\dots,x_{d+1}]$ are the forms defining $F$ in $\P^{n+c}$
and $g_i \in \O_C(U)[x_0,\dots,x_{d+1}]$ are forms of the same degrees in the variables $x_i$.
Let $g_i^0 \in k[x_0,\dots,x_{d+1}]$ be the specialization of $g_i$ at the point $p \in B$. 
By computing the Jacobian ideal, we see that $W$ is singular
along the set $(\Sing(F) \cap \{ g_1^0 = \dots = g_c^0 = 0\}) \times \{p\}$, 
and this set is non-empty if $\dim \Sing(F) \ge c$. 
\end{proof}

By imposing an additional condition on the restriction map on Picard groups, 
we obtain the following classification result for projective bundles and quadric fibrations. 

\begin{theorem}
\label{cor:scroll-quadric}
Let $X$ be a smooth complex projective variety
and $Y \subset X$ a smooth ample subvariety of codimension $r$.  
Assume that the restriction map $\Pic(X) \to \Pic(Y)$ is surjective. 
Let $\p \colon Y \to Z$ be either 
\begin{enumerate}
\item
\label{eq:proj-bundle}
a projective bundle or 
\item
\label{eq:quadric-fibr}
a quadric fibration with integral fibers,
\end{enumerate}
and assume that $\dim Y - \dim Z > r$. 
Then $\p$ extends uniquely to a morphism $\~\p \colon X \to Z$ which is 
a projective bundle in case \eqref{eq:proj-bundle}, and 
either a projective bundle or a quadric fibration with integral fibers 
in case \eqref{eq:quadric-fibr}.
In both cases, the fibers of $\p$ are linearly embedded in the fibers of $\~\p$.
\end{theorem}

\begin{proof}
By \cref{th:P-fibr-Q-fibr}, $\p$ extends uniquely to a morphism $\~\p \colon X \to Z$. 

We claim that $\p$ has relative Picard number 1.
Otherwise $\p$ is necessarily a $\P^1 \times \P^1$-bundle with trivial monodromy
on the cohomology of the fibers, given by the contraction
of a 2-dimensional face of the Mori cone $\CNE(Y)$. 
In this case, by contracting the 
two extremal rays of this face independently, we obtain two 
$\P^1$-bundles $\s_i \colon Y \to W_i$, $i = 1,2$, where each $W_i$ is a $\P^1$-bundle over $Z$. 
Since in this case $\p$ has relative dimension 2, we have $r=1$ and hence $Y$ is an ample divisor on $X$. 
We can therefore apply \cite[Theorem~1.3]{Liu19}. 
The surjectivity of $\Pic(X) \to \Pic(Y)$ implies that
the cases (iii) and (iv) of the quoted theorem cannot occur. 
By the remaining cases (i) and (ii), we see that
both $\P^1$-bundles $\s_i$ extend to $\P^2$-bundles $\~\s_i\colon X \to W_i$. 
Restricting to a general fiber $G$ of $\~\p$, which is 3-dimensional, this gives two distinct
$\P^2$-bundle structures $G \to \P^1$, which is clearly impossible.

We see by \cite[\href{https://stacks.math.columbia.edu/tag/02K4}{Lemma~02K4}]{stacks-project}
that $Z$ is smooth, since $Y$ is smooth and the fibers of $\p$ are reduced.
Therefore there exists a line bundle $\cH$ on $Y$ inducing a linear polarization on the fibers of $\p$. 
By our hypothesis on the Picard groups, 
we can pick a line bundle $\cL$ on $X$ such that $\cL|_Y \cong \cH$. 
The same condition on the Picard groups implies that $\~\p$, like $\p$, has relative
Picard number 1, and therefore
$\cL$ is relatively ample. After twisting by the pull-back of a sufficiently ample
line bundle on $Z$, we can assume that $\cL$ is an ample line bundle. 

Arguing as in the proof of \cite[Theorem~5.8]{BdFL08}, we see that
$\~\p$ is equidimensional with integral fibers. 
Since the setting here is slightly different, we sketch the argument. 
Note that $\p$ is equidimensional, say of relative dimension $n$, and the general fiber of $\~\p$
has dimension $n+r$. 
If $G_i$ is any irreducible component of $G$, then $Y \cap G_i \ne \emptyset$ 
by the ampleness of $Y$ in $X$, and since $Y$ is locally complete intersection of
codimension $r$ in $X$, it follows that $\dim(Y \cap G_i) \ge \dim G_i - r$. 
Therefore $G_i$ has dimension $n+r$ and $F \subset G_i$. 
Note, in particular, that $G$ is regularly embedded in $X$ since
$Z$ is smooth and $\codim(G,X) = \dim Z$, and
therefore it has no embedded components by \cite[Theorem 17.6]{Mat89}. 
Since $\O_{X,F}$ is a regular local ring with a regular sequence locally defining $G$ 
forming part of a regular system of parameters, $\O_{G,F}$ is a regular local ring. 
As every irreducible component of $G$ contains $F$, it follows that $G$ is integral.

Let $m = n+r$ denote the dimension of the fibers of $\~\p$. 
Let $G$ be a smooth fiber of $\~\p$, let $F \subset G$ be the corresponding fiber of $\p$, 
and let $C \subset F$ be a line. By adjunction formula, we have
\[
(K_G + a \,c_1(\cL|_G))\.C = 
(K_F + a \,c_1(\cH|_F) - c_1(\cN_{F/G}))\.C
\]
for any integer $a$. 
Since $\cN_{F/G} = \cN_{Y/X}|_F$ is an ample vector bundle of rank $r$, we see 
that the nef value of $(G,\cL|_G)$ is at least $m + 1$ in case \eqref{eq:proj-bundle}, 
and at least $m$ in case \eqref{eq:quadric-fibr}.
We can therefore apply the main result of \cite{Ion86} (see also \cite{Fuj92}). 
In case \eqref{eq:proj-bundle}, this implies that $(G,\cL|_G) \cong (\P^m,\O_{\P^m}(1))$. 
In case \eqref{eq:quadric-fibr}, we see that $(G,\cL|_G)$ can either be $(\P^m,\O_{\P^m}(1))$,
$(Q,\O_{Q}(1))$ where $Q \subset \P^{m+1}$ is a smooth quadric hypersurface, or a scroll over a curve. 

The last case can be excluded, as follows. 
Assume that $G$ is a scroll over $\P^1$. First, note that $n \ge 2$, and since $F$ is ample in $G$, the map
$\Pic(G) \to \Pic(F)$ is injective by \cref{th:Pic}. Therefore $F \cong \P^1 \times \P^1$. 
Since $\p$ has relative Picard number 1, $Z$ cannot be a point. 
Let $B \subset Z$ be a general complete intersection curve, and let $W = \p^{-1}(B)$.
If $\p|_W \colon W \to B$ is a smooth fibration, then, arguing as at the beginning of the
proof, we see that the monodromy action on $N_1(F)$ must swap the two rulings in
the fibers of $\p|_W$. We claim that the same happens even if $\p$ has some singular fibers. 
Suppose this is not the case. Let $C \subset F$ be a line. 
By taking a general one-parameter deformation of $C$ in $W$, we construct a divisor $D$ on $W$ which
is Cartier since, by Bertini, we can assume that $W$ is smooth. 
If the monodromy acts trivially on $N_1(F)$, then $D$ intersects $F$ into a finite
union of lines in the same ruling of $C$. This implies that $D\.C = 0$, 
and hence $D$ cannot be relatively ample (or antiample) over $B$. 
Since on the other hand $D$ is not numerically trivial over $B$, 
as it intersect positively any line in the other ruling of $F$,
this contradicts the fact that $\p|_W$, having singular fibers, has relative Picard number 1.
Therefore the monodromy action on $N_1(F)$ cannot be trivial and must
swap the two rulings. 
Now, the map $N_1(F) \to N_1(G)$ sends one of the extremal rays of the Mori cone $\CNE(F)$ to 
the extremal ray $R$ of $\CNE(G)$ defining the projective bundle fibration $G \to \P^1$.
The contradiction follows by observing that the monodromy action on $N_1(G)$ must stabilize the ray $R$
since, for dimensional reasons, $G$ cannot have two distinct fibrations to $\P^1$. 
Therefore this case cannot occur, hence
we conclude that $(G,\cL|_G)$ can only be either $(\P^m,\O_{\P^m}(1))$ or $(Q,\O_{Q}(1))$. 

Note that $\~\p$ is flat, see \cref{l:flat-extension} below.
To finish the proof, we apply semi-continuity of
the $\Delta$-genus along the fibers of $\~\p$ \cite[Theorem~5.2]{Fuj75}
and the classification of polarized varieties with $\Delta$-genus zero \cite[Theorems~2.1 and~2.2]{Fuj75},
as in the proof of \cite[Theorem~5.8]{BdFL08}.
This allows us to conclude that all fibers of $\~\p$ are projective spaces or quadric hypersurfaces, 
depending of the situation. 
The sheaf $\~\pi_*\cL$ is locally free on $Z$, the surjection $\~\pi^*\~\pi_*\cL\to \cL$ 
gives the desired linear embedding $X\inj \P(\~\pi_*\cL)$ that gives $X$ the projective bundle 
or quadric fibration structure, and the surjection $\~\pi_*\cL \to \p_*\cH$
gives the fiberwise linear embedding of $Y$ into $X$. 
\end{proof}

\begin{lemma}
\label{l:flat-extension}
Let $X$ be a complex projective Cohen--Macaulay variety and $Y \subset X$ a regularly embedded 
ample subvariety. Let $\f \colon X \to Z$ be a morphism with $Z$ smooth. 
If $\f|_Y \colon Y \to Z$ is flat, then so is $\f$. 
\end{lemma}

\begin{proof}
By \cite[\href{https://stacks.math.columbia.edu/tag/00R4}{Lemma 00R4}]{stacks-project}, 
it suffices to show that $\f$ is equidimensional. 
The flatness of $\f|_Y$ implies that the map is surjective and equidimensional. 
Each irreducible component of any fiber of $\f$ must intersect $Y$, by ampleness of $Y$,
and the fact that $Y$ is regularly embedded in $X$ implies that such intersection 
will be of codimension $\le \codim(Y,X)$ in the given component.
This forces $\f$ to be equidimensional. 
\end{proof}


\end{document}